\documentclass[12pt,a4paper]{article}

\usepackage{amsfonts}
\usepackage{amsmath}
\usepackage{amsbsy}
\usepackage{amsxtra}
\usepackage{latexsym}
\usepackage{amssymb}
\makeatletter
\g@addto@macro\thesection.
\makeatother

\DeclareMathOperator{\cone}{cone}

\begin{document}
\title{Isotonicity of the projection onto the monotone cone
\thanks{{\it 1991 A M S Subject Classification.} Primary 90C33;
Secondary 15A48, {\it Key words and phrases.} Metric projection onto convex cones}}
\author{A. B. N\'emeth\\Faculty of Mathematics and Computer Science\\Babe\c s Bolyai University, Str. Kog\u alniceanu nr. 1-3\\RO-400084 Cluj-Napoca, Romania\\email: nemab@math.ubbcluj.ro \and S. Z. N\'emeth\\School of Mathematics, The University of Birmingham\\The Watson Building, Edgbaston\\Birmingham B15 2TT, United Kingdom\\email: nemeths@for.mat.bham.ac.uk}
\date{}
\maketitle

\begin{abstract}
A wedge (i.e., a closed nonempty set in the Euclidean space stable under addition
and multiplication with non-negative scalars) induces by a standard way a semi-order 
(a reflexive and transitive binary relation) in the space. 
The wedges admitting isotone metric projection with respect to the semi-order
induced by them are characterized. The obtained result is used to show that the monotone
wedge (called monotone cone in regression theory) admits isotone projection.
\end{abstract}

\newenvironment{proof}{{\bf Proof.}}{\hfill$\Box$\\}
\newtheorem{theorem}{Theorem}
\newtheorem{definition}{Definition}
\newtheorem{lemma}{Lemma}
\newtheorem{proposition}{Proposition}
\newtheorem{remark}{Remark}
\newtheorem{example}{Example}
\newtheorem{corollary}{Corollary}

\newcommand{\rem}{{\bf  \hspace*{2 mm} Remark\ }}
\newcommand{\defi}{{\bf  \hspace*{2 mm} Definition\ }}
\newcommand{\arr}{\longrightarrow}
\newcommand{\Arr}{\Rightarrow}
\newcommand{\R}{\mathbb R}
\newcommand{\N}{\mathbb N}
\newcommand{\al}{\alpha}
\newcommand{\la}{\lambda}
\newcommand{\lang}{\langle}
\newcommand{\rang}{\rangle}

\newcommand{\conv}{\textrm{co}}

\newcommand{\stm}{\setminus}
\newcommand{\sbs}{\subset}
\newcommand{\mc}{\mathcal}
\newcommand{\ri}{\textrm{ri}}
\newtheorem{theo}{Theorem}
\newtheorem{lem}{Lemma}
\newtheorem{cor}{Corollary}
\newtheorem{prp}{Proposition}
\newcommand{\edem}{$\hspace{\stretch{22}}\Box$\vskip3mm}

\section{Introduction}
The metric projection onto convex cones is an important tool in solving problems in
metric geometry, statistics, image reconstruction etc.
The idea to relate the ordering induced by the convex cone and the metric projection
onto the convex cone goes back to the paper \cite{IsacNemeth1986} of G. Isac and A. B. N\'emeth, 
where a convex cone in the Euclidean space which admits an isotone projection 
onto it (called by the authors \emph{isotone projection cone}) was characterized. The isotonicity is 
considered with respect to the order induced by the convex cone.
This notion was considered in the context of the complementarity theory where
the isotonicity of the projection provides new existence results and iterative methods
\cite{IsacNemeth1990b,IsacNemeth1990c,Nemeth2009a}.

It turns out that the isotonicity
of the projection is a very strong requirement which implies the latticiality of the order induced by the convex cone.
Thus, the investigation of the isotone projection cones becomes part of the theory of
latticially ordered Euclidean and Hilbert spaces.

 A
simple finite method of projection onto isotone projection cones proposed by us (see
\cite{NemethNemeth2009}) has become important in the effective handling of
all the problems involving projection onto these cones. Besides nonlinear
complementarity, isotone projection cones have applications in
other domains of optimization theory. The positive monotone convex cone used in the Euclidean
distance geometry (see \cite{Dattorro2005}) is an isotone projection one. 
Our method has become important in the effective handling of 
the problem of map-making from relative distance information e.g., stellar cartography 
(see 
{\small
\begin{verbatim}
www.convexoptimization.com/wikimization/index.php/Projection_on_Polyhedral_Cone 
\end{verbatim}}

\noindent and Section 5.13.2.4 in \cite{Dattorro2005}).

Although we shall not consider projection methods in this note,
some of the results developed in \cite{NemethNemeth2009}
will be useful in our proofs. The notion of the cone in the
above cited papers is used in the sense of ``closed convex pointed cone''.
Confronted with the question if the so called monotone cone (which is in fact a wedge in our 
terminology) used in regression theory admits or not an isotonic metric
projection onto it (where isotonicity is considered with respect to the semi-order the
monotone cone introduces), we shall develop a general theory in order to apply it to this 
special case. This seems to be the simplest way to tackle this problem. By using this 
approach, it turns out that the monotone cone indeed admits an isotone metric projection onto 
it.

\section{Projecting onto closed wedges in $\R^m$}

If $C$ is a non-empty, closed convex set in $\R^m$, then for
each $x\in \R^m$ there exists a unique nearest point $P_Cx\in C$,
that is, a point with the property that
$$\|x-P_Cx\|=\inf \{\|x-c\|:\;c\in C\},$$
where $\|.\|$ stands for the Euclidean norm in $\R^m$ (\cite{Zarantonello1971}).

The mapping $P_C:\R^m\to C$ is called \emph{the nearest point mapping of $\R^m$
onto $C$} or simply the \emph{the (metric) projection onto $C$}.

Let $W$ be a \emph{wedge} in $\R^m$, i. e., a closed nonempty set with
(i) $W+W\subset W$ and (ii) $tW\subset W,\;\forall \;t\in \R_+ =[0,+\infty)$. If 
$W\cap (-W)=\{0\}$, then $W$ is called a \emph{cone}.

\begin{lemma}\label{ossz}
Suppose $L=W\cap (-W)$ (the maximal subspace contained in $W$) and let $L^\perp$ be
its orthogonal complement. Denote $K=L^\perp \cap W$. Then $K$ is a cone in $L^\perp$,
\begin{equation}\label{ortossz}
W=K\oplus L
\end{equation}
where $\oplus$ stands for the orthogonal sum, and
\begin{equation}\label{projek}
P_Wx=P_Kx_k+x_l
\end{equation}
where $x=x_k+x_l$ with $x_l\in L$ and $x_k\in L^\perp.$
\end{lemma}

\begin{proof}
The relation (\ref{ortossz}) follows directly  from
$$W=W\cap (L^\perp \oplus L).$$
It is known (\cite{Zarantonello1971}) that the projection $P_Wx$ of $x$ onto 
the wedge $W$ is characterized
by the couple of relations:
\begin{equation}\label{elso}
\lang x-P_Wx,y\rang \leq 0,\;\forall \; y \in W,
\end{equation}
and
\begin{equation}\label{masod}
\lang x-P_Wx,P_Wx\rang = 0.
\end{equation}
Hence, we have to verify the above relations for $P_Kx_k+x_l$ instead of $P_Wx$.

By the relation (\ref{ortossz}) $P_Kx_k+x_l \in W.$

Take an arbitrary $y\in W$ represented by (\ref{ortossz}) in the form
$$y=y_k+y_l$$
with $y_k\in K$ and $y_l\in L$.
Then we have
$$ \lang x_k+x_l-(P_Kx_k+x_l),y_k+y_l\rang = \lang x_k-P_Kx_k,y_k\rang \leq 0,\;\;\forall \;y=y_k+y_l\in W,$$
because $y_l$ is perpendicular to $x_k-P_Kx_k\in L^\perp,$
and because of the relation similar to (\ref{elso}) characterizing the projection of $x_k$
onto the cone $K$ in $L^\perp.$
Thus, relation (\ref{elso}) holds for $P_Kx_k+x_l$ in place of $P_Wx.$

We further have 
$$\lang x_k+x_l-(P_Kx_k+x_l),P_Kx_k+x_l\rang =\lang x_k-P_Kx_k,P_Kx_k\rang =0$$
because $x_l$ is perpendicular to $x_k-P_Kx_k$ and because of the relation similar to
(\ref{masod}) applied to $x_k\in L^\perp$ and its projection onto $K$.

The obtained relation is exactly (\ref{masod}) for $P_Kx_k+x_l$ instead of $P_Wx.$
\end{proof}

\section{The isotonicity of the projection onto a closed wedge in $\R^m$}

By putting $u\leq_W v$ whenever $u,\;v\in \R^m$ and $v-u\in W,$
the wedge $W\subset \R^m$ induces a \emph{semiorder} $\leq_W$ in $\R^m$
which is translation invariant (i. e. $u\leq_W v$ implies $u+z\leq_W v+z$ for any $z\in \R^m$)
and scale invariant (i.e. $u\leq_W v$ implies $tu\leq_W tv$ for any $t\in \R_+$).

The projection $P_W$ is said \emph{$W$-isotone} if $u,\;v\in \R^m,\;u\leq_Wv$ implies $P_Wu\leq_W P_Wv.$
If $P_W$ is $W$-isotone, then $W$ is called an \emph{isotone projection wedge}. A cone $K$ is 
called an \emph{isotone projection cone} if it is an isotone projection wedge.

\begin{theorem}\label{fo}
Let $W\subset \R^m$ be a wedge,
$$W=K\oplus L$$
with $L=W\cap (-W)$ and $K=W\cap L^\perp$.
Then $W$ is an isotone projection wedge if and only if
$K\subset L^\perp$ is an isotone projection cone in $L^\perp$.
\end{theorem}

\begin{proof}
Take $u,\;v\in L^\perp$. Then, $u\leq_K v$ is equivalent to $u\leq_W v$.
If $P_W$ is $W$-isotone, then $u\leq_W v$ implies by Lemma \ref{ossz}
\begin{equation*}
P_Kv-P_Ku=P_Wv-P_Wu\in W.
\end{equation*}
Since $P_Ku,\:P_Kv\in L^\perp$, it follows that
$$P_Kv-P_Ku\in L^\perp \cap W=K.$$
The obtained relation shows that $P_K$ is $K$-isotone,
concluding the proof of the necessity of the theorem.

Suppose now that $P_K$ is $K$-isotone and take $u,\;v\in \R^m$
with $u\leq_W v$. If $u=u_k+u_l$ and $v=v_k+v_l$ with
$u_k,\;v_k\in L^\perp,$ and $u_l,\;v_l\in L,$
then using formula (\ref{ortossz})
$$v-u=v_k-u_k+v_l-u_l\in K\oplus L$$
and hence $v_k-u_k\in K$, that is $u_k\leq_K v_k$
and by the $K$-isotonicity of $P_K$ it follows that
$$P_Kv_k-P_Ku_k\in K.$$
Hence, using formula (\ref{projek}) we have
$$P_Wv-P_Wu=P_Kv_k+v_l-P_Ku_k-u_l= P_Kv_k-P_Ku_k +v_l-u_l\in K\oplus L=W.$$
That is $P_Wu\leq_W P_Wv$, which concludes the isotonicity of $P_W$.
\end{proof}

A simple geometric characterization of the isotone projection
cones was given in \cite{IsacNemeth1986}. It uses the notion of
the polar of a wedge.

If $W\subset \R^m$ is a wedge, then the set
$$W^{\perp}=\{y\in \R^m:\;\langle x,y\rangle \leq 0, \;\forall \;x\in W\},$$
is called the \emph{polar} of the wedge  $W$. The set
$W^{\perp}$ is obviously a wedge. If the wedge $W$ is
\emph{generating} in the sense that $W-W=\R^m$, then
the polar $W^{\perp}$ is a cone. 

We have the following
easily verifiable result:
 \begin{lemma}\label{mineknev}
Suppose that $W$ is a generating wedge. Using the notations
introduced in the Theorem \ref{fo}, and denoting the polar of the cone $K$
in the subspace $L^{\perp}$ by $K^{\perp}$, we have the relation
$$W^\perp=iK^{\perp},$$
where $i$ is the inclusion mapping of $L^{\perp}$ into $\R^m.$
\end{lemma}
Putting together the main result in \cite{IsacNemeth1986}, Theorem
\ref{fo} and Lemma \ref{mineknev}, we have the following conclusion:
\begin{corollary}
The generating wedge $W$ is an isotone projection wedge if and only if
its polar $W^{\perp}$ is a cone generated by linearly independent
vectors forming mutually non-acute angles. 
\end{corollary}

\section{Application: The isotonicity of the monotone wedge}

Suppose that $\R^m$ is endowed with a Cartesian coordinate system,
and $x\in \R^m$, $x=(x^1,...,x^m)$ where $x^i$ are the coordinates of
$x$ with respect to this reference system. The set
\begin{equation}\label{mon}
W=\{x\in \R^m:\;x^1\geq x^2\geq ...\geq x^m\}
\end{equation}
is called the \emph{monotone cone} (see e.g.\cite{BestChakravarti1990}).
To be in accordance with our earlier terminology, we shall use for $W$
instead the term \emph{monotone wedge}.

 Let 
\begin{equation*}
L=W\cap(-W)=\{x\in \R^m:\;x^1=x^2=...=x^m\}.
\end{equation*}
Then $L\subset W$, the maximal subspace contained in $W$, is of dimension one.
We have also that
\begin{equation}\label{subcone}
K=L^\perp \cap W
\end{equation}
is an $m-1$-dimensional cone in the hyperplane $L^\perp$ and
\begin{equation*}
W=W\cap (L^\perp \oplus L)=K\oplus L.
\end{equation*}
We will show that the cone $K$ given by (\ref{subcone}) is
an isotone projection cone in $L^\perp$. To do this,
we have to introduce some notations.

Let us take the following base in $\R^m$: 
$$e_1=(1,0,...,0)$$
$$e_2=(1,1,0,...,0),$$
$$\dots$$
$$e_{m-1}=(1,...,1,0),$$
$$e_m=(1,1,...,1).$$
An arbitrary element $x=(x^1,...,x^m)\in \R^m$ can be represented in the form
\begin{equation}\label{kifejez}
x=(x^1-x^2)e_1+(x^2-x^3)e_2+...+(x^{m-1}-x^m)e_{m-1}+x^me_m,
\end{equation}
the relation $x\in W$ being equivalent with 
\begin{equation}\label{ekelem}
x^{j-1}-x^j\geq 0,\;\;j=2,...,m.
\end{equation}

Let us consider further the following base in $L^\perp$:

$$e'_1=(m-1,-1,-1,...,-1).$$
$$e'_2=(m-2,m-2,-2,...,-2),$$

$$\dots$$
$$e'_{m-1}=(1,...,1,-(m-1)).$$

The following notation is standard in the convex geometry
and ordered vector space theory: If $M\subset \R^m$ is a non-empty set,
then let
$$\cone M=\{t^1m_1+...+t^km_k:\;m_i\in M,\;t^i\in \R_+=[0,+\infty),\;i=1,...,k;\;k\in \N\}.$$
(The set $\cone W$ is the minimal wedge containing the set $M$
and it is called \emph{the wedge generated by $M$}.)

We will see next that
\begin{equation}\label{ak}
K=\cone \{e'_1,...,e'_{m-1}\}.
\end{equation}

Since $e'_j\in W\cap L^\perp,$
we have obviously that

\begin{equation}\label{benne}
\cone \{e'_1,...,e'_{m-1}\}\subset K.
\end{equation}

Comparing the vectors $e_i$ and $e'_j$ we get
\begin{equation}\label{comparing}
\frac{1}{m-j+1}(e'_j+e_m)=e_j,\;\;j=1,...,m-1. 
\end{equation}
By substitution of $e_j,\;j=1,...,m-1$, the representation
(\ref{kifejez}) of $x$ becomes
\begin{equation}\label{kifejez1}
x =(x^1-x^2)\frac{1}{m}(e'_1-e_m)+(x^2-x^3)\frac{1}{m-1}(e'_2+e_m)+...+(x^{m-1}-x^m)\frac{1}{2}(e'_{m-1}+e_m)+x^me_m.
\end{equation}
Suppose now that $x\in W$, that is, relations (\ref{ekelem}) hold.
Then  the coefficients of $e'_j,\;j=1,...,m-1$ in its representation (\ref{kifejez1}) are 
non-negative. Thus, we have
\begin{equation}\label{kesz}
	x\in W \Leftrightarrow x =\sum_{j=1}^{m-1} t^je'_j+t^me_m,\textrm{ }t^j\in \R_+,
	\textrm{ }j=1,...,m-1,\textrm{ }t^m\in \R.
\end{equation}

In particular, if $x\in K$, then, by \eqref{subcone}, we have $x\in L^\perp$. Hence, by 
multiplying (\ref{kesz}) scalarly by $e^m$ and by using $\lang x,e_m\rang=0$ 
(which follows from $x\in L^\perp$ and $e_m\in L$) and $\lang e'_j,e_m\rang=0$ 
(which follows from $e'_j\in L^\perp$ and $e_m\in L$), we get $t^m=0.$ 
This reasoning shows that
$$K \subset \cone \{e'_1,...,e'_{m-1}\},$$
inclusion which together with (\ref{benne}) proves (\ref{ak}).

\vspace{4mm}

We consider now the vectors
$$u_1=(-1,1,0,...,0),$$
$$u_2=(0,-1,1,0,...,0),$$
$$\dots$$
$$u_{m-1}=(0,...,0,-1,1).$$

Then $u_i\in L^\perp,\;i=1,...,m-1,$ and we have
\begin{equation}\label{polar}
\lang u_i,e'_j\rang=0 \;\textrm{if} \; i\not= j,\;\;\lang u_i,e'_i\rang <0,\;i,j=1,...,m-1.
\end{equation}

According to the reasonings in \cite{NemethNemeth2009} the
relations (\ref{polar}) show that
$$\cone \{u_1,...,u_{m-1}\}$$
is the polar of $K$ in the subspace $L^\perp$.
Further, we have
$$\lang u_i,u_j\rang \leq 0 \;\;\textrm{if} \; i\not=j.$$
By the main result in \cite{IsacNemeth1986} this
shows that $K$ is an isotone projection cone in $L^\perp.$ 

In conclusion, using Theorem \ref{fo}
we have the

\begin{corollary}
The monotone wedge $W$ given by the formula (\ref{mon}) admits an isotone projection.
\end{corollary}

\end{document}